\documentclass[preprint,12pt]{elsarticle}

\usepackage{amssymb}
\usepackage{amsthm}
\usepackage{amsfonts,amsmath}
\usepackage{mathrsfs}

\newtheorem{theorem}{Theorem}
\newtheorem{proposition}{Proposition}

\newtheorem{lemma}[theorem]{Lemma}

\newcommand*{\la}{\lambda}

\newcommand*{\N}{\mathbb{N}}
\newcommand*{\Q}{\mathbb{Q}}

\newcommand*{\R}{\mathbb{R}}

\newcommand*{\C}{\mathbb{C}}

\newcommand*{\spec}{\mathrm{spec }\,}

\newcommand*{\cH}{\mathcal{H}}

\newcommand*{\cL}{\mathcal{L}}

\newcommand*{\cS}{\mathcal{S}}

\newcommand*{\id}{\mathrm{id}}

\newcommand*{\tr}{\mathrm{tr}}
\newcommand*{\ket}[1]{| #1 \rangle}
\newcommand*{\bra}[1]{\langle #1 |}
\newcommand*{\spr}[2]{\langle #1 | #2 \rangle}

\newcommand*{\proj}[1]{\ket{#1}\bra{#1}}

\newcommand*{\Sym}{\mathrm{Sym}}

\newcommand*{\GL}{\mathsf{GL}}
\newcommand*{\SL}{\mathsf{SL}}
\newcommand*{\SG}{\ensuremath{\mathsf{S}}}

\newcommand{\rect}{\text{\scalebox{1.5}[1]{\ensuremath{\square}}}}
\def\mp{\mathsf{Kron}}
\def\proj{\mathbb{P}}
\def\e{\epsilon}
\def\Id{\mathrm{id}}
\def\ud{{\mathbf u}}

\newcommand{\per}{\mathrm{per}}
\renewcommand{\det}{\mathrm{det}}
\newcommand{\GLC}[1]{\ensuremath{\mathsf{GL}_{#1}(\C)}}
\newcommand{\Weyl}[1]{\mathscr{V}_{#1}}
\newcommand{\Specht}[1]{\mathscr{S}_{#1}}

\begin{document}

\begin{frontmatter}



\title{Nonvanishing of Kronecker coefficients\\ for rectangular shapes }


\author[PB]{Peter B\"urgisser\fnref{gPB}}
\address[PB]{Institute of Mathematics, University of Paderborn, D-33098 Paderborn, Germany}
\fntext[gPB]{Supported by the German Science Foundation (grant BU 1371/3-1 of the SPP 1388 on Representation Theory)}
\ead{pbuerg@upb.de}

\author[MC]{Matthias Christandl\fnref{gM}}
\address[MC]{Institute for Theoretical Physics, ETH Zurich, CH-8093 Zurich, Switzerland}
\fntext[gM]{Supported by the Swiss National Science Foundation (grant PP00P2-128455) and the German Science Foundation (grant~CH 843/1-1 of the SPP 1388 on Representation Theory and grant CH 843/2-1)}
\ead{christandl@phys.ethz.ch} 

\author[PB]{Christian Ikenmeyer\fnref{gPB}}
\ead{ciken@math.upb.de}


%

\begin{abstract}
We prove that for any partition
$(\la_1,\ldots,\la_{d^2})$ of size~$\ell d$
there exists $k\ge 1$ such that
the tensor square of the irreducible representation
of the symmetric group $\SG_{k\ell d}$ with respect to the
rectangular partition $(k\ell,\ldots,k\ell)$ contains
the irreducible representation
corresponding to the stretched partition
$(k\la_1,\ldots,k\la_{d^2})$.
We also prove a related approximate version of this
statement in which the stretching factor~$k$ is
effectively bounded in terms of~$d$.
We further discuss the consequences for geometric complexity theory 
which provided the motivation for this work.
\end{abstract}


\end{frontmatter}




\section{Introduction}

Kronecker coefficients are the multiplicities occurring in tensor product decompositions
of irreducible representations of the symmetric groups.
These coefficients play a crucial role in geometric complexity theory~\citep{GCT1,GCT2},
which is an approach to arithmetic versions of the famous P versus NP problem
and related questions in computational complexity via geometric representation theory.
As pointed out in \citep{blmw:09} (see Section~\ref{se:connGCT}),  
for implementing this approach,
one needs to identify certain partitions $\la\vdash_{d^2}\ell d$ with the property that a symmetric version of
the Kronecker coefficient associated with $\la,\rect,\rect$ vanishes, where
$\rect:=(\ell,\ldots,\ell)$ stands for the rectangle partition of length~$d$. 
Computer experiments show that such $\la$ occur rarely.
Our main result confirms this experimental finding.
We prove that for any $\la\vdash_{d^2}\ell d$ there exists a stretching factor~$k$
such that the Kronecker coefficient of $k\la,k\rect,k\rect$ is nonzero (Theorem~\ref{th:main1}). 
Here, $k\la$ stands for the partition arising by multiplying all components of $\la$ by~$k$.
We also prove a related approximate version of this statement
(Theorem~\ref{th:main2}) that suggests that the stretching factor~$k$ may
be chosen not too large. Similar results are shown to hold for the symmetric version of the Kronecker coefficient and thus have a bearing on geometric complexity theory (see Lemma~\ref{pro:sKC} and Section~\ref{se:connGCT}).

Our proof relies on a recently discovered connection between Kronecker coefficients
and the spectra of composite quantum states~\citep{Klyachko04,ChrMit06}.
Let $\rho_{AB}$ be the density operator of a bipartite quantum system
and let $\rho_A$, $\rho_B$ denote the density operators corresponding
to the systems $A$ and $B$, respectively. It turns out that
the set of possible triples of spectra
$(\spec\rho_{AB},\spec\rho_{A},\spec\rho_{B})$
is obtained as the closure of the
set of triples $(\overline{\la},\overline{\mu},\overline{\nu})$
of normalized partitions $\la,\mu,\nu$
with nonvanishing Kronecker coefficient,
where we set $\overline{\la}:=\frac1{|\la|}\la$.
For proving the main theorem it is therefore sufficient to construct,
for any prescribed spectrum~$\overline{\la}$, a density matrix
$\rho_{AB}$ having this spectrum and such that the spectra of
$\rho_{A}$ and $\rho_{B}$ are uniform distributions.

The set of possible triples of spectra
$(\spec\rho_{AB},\spec\rho_{A},\spec\rho_{B})$
is interpreted in \citep{Klyachko04} 
as the moment polytope
of a complex algebraic group variety,
thus linking the problem to geometric invariant theory.
We do not not use this connection in our paper. Instead
we argue as in \cite{ChrMit06} using the
estimation theorem of~\cite{kewe:01}.
The exponential decrease rate in this estimation allows
us to derive the bound on the
stretching factor in Theorem~\ref{th:main2}.

\section{Preliminaries}

\subsection{Kronecker coefficients and their moment polytopes}

A {\em partition} $\la$ of $n\in\N$ is a monotonically decreasing sequence
$\la =(\la_1, \la_2, \ldots)$ of natural numbers such that $\lambda_i=0$
for all but finitely many~$i$.
The length $\ell(\la)$ of~$\la$ is defined as the number of its nonzero parts
and its size as $|\la|:=\sum_i\lambda_i$.
One writes $\la\vdash_d n$ to express that $\lambda$ is a partition of $n$
with $\ell(\la)\le d$.
Note that $\bar\la:= \la/n=(\la_1/n, \la_2/n, \ldots)$ defines
a probability distribution on~$\N$.

It is well known~\citep{fuha:91} that the complex irreducible representations
of the symmetric group $\SG_n$ can be labeled by partitions $\la\vdash n$ of $n$.
We shall denote by $\Specht{\la}$ the irreducible representation of $\SG_n$
associated with $\la$.
The {\em Kronecker coefficient} $g_{\la,\mu,\nu}$ associated with
three partitions $\la,\mu,\nu$ of $n$ is defined as the dimension of
the space of $\SG_n$-invariants in the tensor product
$\Specht{\la}\otimes \Specht{\mu}\otimes \Specht{\nu}$.
Note that $g_{\la,\mu,\nu}$ is invariant with respect to a permutation of the partitions.
It is known that $g_{\la,\mu,\nu}=0$ vanishes
if $\ell(\la) > \ell(\mu)\ell(\nu)$.
Equivalently, $g_{\la,\mu,\nu}$
may also be defined as the multiplicity of $\Specht{\la}$
in the tensor product $\Specht{\mu}\otimes \Specht{\nu}$.
If $\mu=\nu$ we define the 
{\em symmetric Kronecker coefficient} $sg^\la_{\mu}$ as the
multiplicity of $\Specht{\la}$ in the symmetric square 
$\Sym^2(\Specht{\mu})$. We note that 
$sg^\la_{\mu} \le g_{\la,\mu,\mu}$. 

The Kronecker coefficients also appear when studying representations of
the general linear groups $\GL_d$ over $\C$.
We recall that rational irreducible $\GL_d$-modules 
are labeled by their highest weight, a monotonically decreasing list
of $d$~integers, cf.\ Fulton and Harris~\cite{fuha:91}.
We will only be concerned with highest weights consisting of
nonnegative numbers, which are therefore of the form $\la\vdash_d k$ for modules of degree~$k$.
We shall denote by ${\Weyl \la}$ the irreducible $\GL_d$-module with highest weight~$\la$.

Suppose now that $\la\vdash_{d_1 d_2} k$.
When restricting with respect to the morphism
$
\GL_{d_1}\times \GL_{d_2}\to \GL_{d_1d_2},(\alpha,\beta)\mapsto \alpha\otimes\beta,
$
then the module~${\Weyl \la}$ splits as follows:
\begin{equation}\label{eq:split}
{\Weyl \la} = \bigoplus_{\mu \vdash_{d_1} k ,\nu\vdash_{d_2} k} g_{\la,\mu,\nu} {\Weyl \mu}\otimes {\Weyl \nu} .
\end{equation}

Even though being studied for more than fifty years,
Kronecker coefficients are only understood in some special cases.
For instance, giving a combinatorial interpretation of the numbers~$g_{\la,\mu,\nu}$
is a major open problem, cf.\ Stanley~\cite{stan:99,stan:00} for more information.

We are mainly interested in whether $g_{\la,\mu,\nu}$ vanishes or not.
For studying this in an asymptotic way one may consider, for fixed
$d=(d_1,d_2,d_3)\in\N^3$ with $d_1\le d_2\le d_3 \le d_1d_2$, the set
$$
\mp(d) := \Big\{ \frac1{n} (\la,\mu,\nu) \mid n \in \N,
 \la\vdash_{d_1} n, \mu\vdash_{d_2} n, \nu\vdash_{d_3} n\ g_{\la, \mu, \nu} \ne 0 \Big \} .
$$
It turns out that $\mp(d)$ is a rational polytope in $\Q^{d_1+d_2+d_3}$.
This follows from general principles from geometric invariant theory,
namely $\mp(d)$ equals the {\em moment polytope} of the projective variety
$\proj(\C^{d_1}\otimes\C^{d_2}\otimes\C^{d_3})$ with respect to the standard action of
the group $\GL_{d_1}\times\GL_{d_2}\times\GL_{d_3}$,
cf.~\citep{mani:97,fran:02,Klyachko04}.
For an elementary proof that $\mp(d)$ is
a polytope see \cite{ChristHarrowMitch07}.

\subsection{Spectra of density operators}

Let $\cH$ be a $d$-dimensional complex Hilbert space and denote by $\cL(\cH)$
the space of linear operators mapping~$\cH$ into itself.
For $\rho \in \cL(\cH)$
we write $\rho \geq 0$  to denote that $\rho$
is positive semidefinite.
By the {\em spectrum} $\spec\rho$ of~$\rho$ we will understand
the vector $(r_1,\ldots,r_d)$ of eigenvalues of $\rho$
in decreasing order, that is, $r_1\ge \cdots \ge r_d$.
The set of {\em density operators} on $\cH$ is defined as
$$
\cS(\cH):= \{\rho \in \cL(\cH) \mid \rho \geq 0, \tr \rho =1\}.
$$
Density operators are the mathematical formalism to describe the states of quantum objects.
The spectrum of a density operator is a probability distribution on
$[d]:=\{1,\ldots,d\}$.

The state of a system composed of particles $A$ and $B$
is described by a density operator on a tensor product of two Hilbert spaces,
$\rho_{AB}\in\cL(\cH_A\otimes\cH_B)$.
The partial trace
$\rho_A = \tr_B(\rho_{AB})\in\cL(\cH_A)$
of $\rho_{AB}$ obtained by tracing over~$B$
then defines the state of particle $A$.
We recall that the {\em partial trace} $\tr_B$ is the linear map
$\tr_B\colon\cL(\cH_A\otimes\cH_B)\to\cL(\cH_A)$
uniquely characterized by the property
$\tr (R\, \tr_B(\rho_{AB})) =\tr(\rho_{AB} R \otimes \Id )$
for all $\rho_{AB}\in  \cL(\cH_A\otimes \cH_B)$ and $R\in\cL(\cH_A)$.

\subsection{Admissible spectra and Kronecker coefficients}\label{se:speK}

The {\em quantum marginal problem} asks for a description of
the set of possible triples of spectra $(\spec\rho_{AB},\spec\rho_{A},\spec\rho_{B})$
for fixed $d_A=\dim\cH_A$ and $d_B=\dim\cH_B$.
In~\citep{ChrMit06,Klyachko04,ChristHarrowMitch07} it was shown that this set
equals the closure of the moment polytope for Kronecker coefficients, so
\begin{equation*}
\overline{\mp(d_A,d_B,d_Ad_B)} = \Big\{(\spec\rho_{AB},\spec\rho_{A},\spec\rho_{B}) \mid
 \rho_{AB}\in\cL(\cH_A\otimes\cH_B) \Big\}.
\end{equation*}
We remark that this result is related to {\em Horn's problem} that asks for the compatibility conditions
of the spectra of Hermitian operators $A$, $B$, and $A+B$ on finite dimensional Hilbert spaces.
\citet{klya:98} gave a similar characterization of these triples of spectra
in terms of the Littlewood-Richardson coefficients.
The latter are the multiplicities occurring in tensor products of irreducible
representations of the general linear groups.
For Littlewood-Richardson coefficients one can actually avoid the asymptotic description
since the so called saturation conjecture is true~\citep{knta:99}.

\subsection{Estimation theorem}

We will need a consequence of the estimation theorem of~\cite{kewe:01}.
The group $\SG_k\times\GL_d$ naturally acts on the tensor power~$(\C^d)^{\otimes k}$.
Schur-Weyl duality describes the isotypical decomposition of this
module as
\begin{equation}\label{eq:isotyp}
 (\C^d)^{\otimes k} = \bigoplus_{\la\vdash_d k} \Specht{\la}\otimes {\Weyl \la} .
\end{equation}
We note that this is an orthogonal decomposition with respect
to the standard inner product on $(\C^d)^{\otimes k}$.
Let $P_\la$ denote the orthogonal projection of
$(\C^d)^{\otimes k}$ onto $\Specht{\la}\otimes {\Weyl \la}$.
The estimation theorem of Keyl and Werner~\citep{kewe:01} states that
for any density operator $\rho\in\cL(\C^d)$ with spectrum~$r$
we have
\begin{equation}\label{eq:est}
\tr (P_\la\,\rho^{\otimes k}) \le (k+1)^{d(d-1)/2}\exp\big(-\frac{k}2 \|\overline{\la} -r\|_1^2\big)
\end{equation}
(see \citep{ChrMit06} for a simple proof).
This shows that the probability distribution
$\overline{\la}\mapsto\tr (P_\la\,\rho^{\otimes k})$ is
concentrated around $r$ with exponential decay
in the distance $\|\overline{\la}-r\|_1$.

\section{Main results}

By a decreasing probability distribution $r$ on $[d^2]$
we understand $r\in\R^{d^2}$ such that
$r_1\ge\cdots \ge r_{d^2}\ge 0$ and $\sum_i r_i =1$.
We denote by
$\ud_d=(\frac1d,\ldots,\frac1d)$ the uniform probability distribution on $[d]$.

\begin{theorem}\label{th:main1}
The following statements are true: 

{\bf (1)}\
For all decreasing probability distributions~$r$ on $[d^2]$,
the triple $(r,\ud_d,\ud_d)$ is contained in $\overline{\mp(d^2,d,d)}$.

{\bf (2)}\
Let $\la\vdash \ell d$ be a partition into at most $d^2$ parts for $\ell,d\ge 1$
and let $\rect:=(\ell,\ldots,\ell)$ denote the rectangular partition of $\ell d$ into $d$ parts.
Then there exists a stretching factor $k\ge 1$ such that
$g_{k\la, k\rect, k\rect} \neq 0$.
\end{theorem}

The next result indicates that the stretching factor~$k$
may be chosen not too large.

\begin{theorem}\label{th:main2}
Let $\la\vdash_{d^2} \ell d$ and $\e>0$.
Then there exists a stretching factor $k=O(\frac{d^4}{\e^2}\log\frac{d}{\e})$
and there exist partitions $\Lambda\vdash_{d^2} k\ell d$ and
$R_1,R_2\vdash_d k \ell d$ of $k\ell d$ such that
$g_{k\la, R_1, R_2} \neq 0$ and
$$
\|\Lambda - k\la\|_1 \le \e |\Lambda | ,\quad
\,\|R_i - k\rect\|_1 \le \e |R_i| \quad\mbox{ for $i=1,2$.}
$$
\end{theorem}


Suppose that $g_{\la,\mu,\mu}\ne 0$.
By stretching the partitions $\la,\mu$ with two,  
we can guarantee that the corresponding symmetric Kronecker coefficients
does not vanish either.

\begin{lemma}\label{pro:sKC}
Let $\la, \mu  \vdash n$. 
If $\Specht{\la}$ occurs in $\Specht{\mu}\otimes \Specht{\mu}$, then 
$\Specht{2\la}$ occurs in $\Sym^2(\Specht{2\mu})$. In other words, $g_{\lambda, \mu, \mu}\neq 0$ implies $sg_{2\mu}^{2\lambda}\neq 0$.
\end{lemma}

This lemma, when combined with Theorems~\ref{th:main1} and~\ref{th:main2},
shows that finding partitions $\la$ with $sg^{\la}_{\rect}=0$,
as required for the purposes of geometric complexity theory (see below),
requires a careful search.

\section{Connection to geometric complexity theory}\label{se:connGCT}

The most important 
problem of algebraic complexity theory is
Valiant's Hypothesis~\citep{vali:79,vali:82},
which is an arithmetic analogue
of the famous P versus NP conjecture
(see \citep{bucs:96} for background information).
Valiant's Hypothesis can be easily stated in precise mathematical terms.

Consider the determinant $\det_d = \det [x_{ij}]_{1\le i,j\le d}$
of a $d$ by $d$ matrix of variables $x_{ij}$, 
and for $m<d$, 
the {\em permanent} of its $m$ by $m$ submatrix defined as
$$
\per_m := \sum_{\sigma\in S_m} x_{1,\sigma(1)}\cdots x_{m,\sigma(m)} .
$$
We choose $z:=x_{dd}$ as a homogenizing variable and view
$\det_d$ and $z^{d-m}\per_m$
as homogeneous functions $\C^{d^2}\to \C$ of degree~$d$.
How large has $d$ to be in relation to~$m$ such that there is
a linear map $A\colon \C^{d^2}\to\C^{d^2}$ with the property that
\begin{equation}\tag{*}
z^{d-m}\per_m = \det_d \circ A ?
\end{equation}
It is known that such $A$ exists for $d=O(m^2 2^m)$.
Valiant's Hypothesis states that (*)~is
impossible for $d$ polynomially bounded in~$m$.

\citet{GCT1} suggested
to study an orbit closure problem related to~(*).
Note that the group $\GL_{d^2} = \GLC {d^2}$ acts on the space
$\Sym^d (\C^{d\times d})^*$
of homogeneous polynomials of degree~$d$
in the variables $x_{ij}$ by substitution.
Instead of (*), we ask now whether  
\begin{equation}\tag{**}
 z^{d-m}\per_m\in \overline{\GL_{d^2}\cdot\det_{d}} .
\end{equation}
\citet{GCT1} conjectured that 
(**)~is impossible for $d$ polynomially bounded in $m$, which would imply 
Valiant's Hypothesis. 

Moreover, in~\citep{GCT1,GCT2} it was proposed to show
that (**) is impossible for specific values $m,d$ by
exhibiting an irreducible representation of $\SL_{d^2}$
in the coordinate ring of the orbit closure of $z^{d-m}\per_m$,
that does not occur in the coordinate ring $\C[\overline{\GL_{d^2}\cdot\det_{d}}]$
of $\overline{\GL_{d^2}\cdot\det_{d}}$.
We call such a representation of $\SL_{d^2}$
an {\em obstruction for (**)} for the values $m,d$.

We can label the irreducible $\SL_{d^2}$-representations 
by partitions $\la$ into at most $d^2-1$ parts:  
For $\la\in\N^{d^2}$ such that 
$\la_1\ge\ldots\ge\la_{d^2-1} \ge \la_{d^2}=0$
we shall denote by $\Weyl\la (\SL_{d^2})$ 
the irreducible $\SL_{d^2}$-representation
obtained from the irreducible $\GL_{d^2}$-representation  
$\Weyl \la$ with the highest weight~$\la$ by restriction. 

If $\Weyl\la (\SL_{d^2})$ is an obstruction for $m,d$, then 
we must have $|\la| = \sum_i \la_i = \ell d$ for some $\ell$, 
see~\cite[Prop. 5.6.2]{blmw:09}. 
We call the representation $\Weyl\la (\SL_{d^2})$ 
a {\em candidate for an obstruction} iff 
$\Weyl\la (\SL_{d^2})$ does not occur in 
$\C[\overline{\GL_{d^2}\cdot\det_{d}}]$. The following proposition relates the search for obstructions to the symmetric Kronecker coefficient.

\begin{proposition}\label{pro:candidate}
Suppose that  $|\la| = \ell d$ and write 
$\rect=(\ell,\ldots,\ell)$ with $\ell$ occurring $d$ times.
Then $\Weyl\la (\SL_{d^2})$ is a candidate for an obstruction iff 
the symmetric Kronecker coefficient $sg^\la_{\rect}$ vanishes. 
\end{proposition}

\begin{proof}
This is an immediate consequence of 
Prop. 4.4.1 and Prop. 5.2.1 in~\citep{blmw:09}.
\end{proof}

We may thus interpret this paper's main results
by saying that candidates for obstructions are rare.

\section{Proofs}
\subsection{Proof of Theorem~\ref{th:main1}}

We know that $\mp(d,d,d^2)$ is a rational polytope, i.e., defined by
finitely many affine linear inequalities with rational coefficients.
This easily implies that a rational point in $\overline{\mp(d,d,d^2)}$
actually lies in $\mp(d,d,d^2)$.
Hence the second part of Theorem~\ref{th:main1}
follows from the first part.

The first part of Theorem~\ref{th:main1} follows from the spectral characterization
of $\overline{\mp(d,d,d^2)}$ described in Section~\ref{se:speK} and the following result.

\begin{proposition}\label{le:densityoperator}
For any decreasing probability distribution $r$ on $[d^2]$ there exists
a density operator $\rho_{AB} \in \cS(\cH_A\otimes\cH_B)$ with spectrum $r$
such that $\tr_A (\rho_{AB}) = \tr_B(\rho_{AB}) = \ud_d$,
where $\cH_A\simeq\cH_B\simeq\C^d$.
\end{proposition}

The proof of Proposition~\ref{le:densityoperator} proceeds by different lemmas.
It will be convenient to use the bra and ket notation of quantum mechanics.
Suppose that $\cH_A$ and $\cH_B$ are $d$-dimensional Hilbert spaces.
We recall first the {\em Schmidt decomposition}:
for any $\ket{\psi}\in\cH_A\otimes\cH_B$,
there exist orthonormal bases $\{\ket{u_i}\}$ of $\cH_A$ and $\{\ket{v_i}\}$ of~$\cH_B$
as well as nonnegative real numbers $\alpha_i$,
called {\em Schmidt coefficients},
such that $\ket{\psi}=\sum_i \alpha_i \ket{u_i}\otimes \ket{v_i}$.
Indeed, the $\alpha_i$ are just the singular values of $\ket{\psi}$
when we interpret it as a linear operator in
$\cL(\cH_A^*,\cH_B)\simeq\cH_A\otimes\cH_B$.

\begin{lemma}\label{le:uno}
Suppose that $\ket{\psi}\in\cH_A\otimes\cH_B$ has the
Schmidt coefficients~$\alpha_i$
and consider $\rho := \ket{\psi}\bra{\psi} \in \cL(\cH_A\otimes\cH_B)$.
Then $\tr_B(\rho)\in \cL( \cH_A)$,
obtained by tracing over the $B$-spaces, has eigenvalues $\alpha_i^2$.
\end{lemma}

\begin{proof}
We have $\ket{\psi}=\sum_i\alpha_i \ket{u_i}\otimes \ket{v_i}$
for some orthonormal bases $\{\ket{u_i}\}$ and $\{\ket{v_i}\}$
of $\cH_A$ and $\cH_B$, respectively. This implies
$$
\rho = \ket{\psi}\bra{\psi} = \sum_{i,j} \alpha_i \alpha_j
 \ket{u_i}\bra{u_j} \otimes \ket{v_i}\bra{v_j}
$$
and tracing over the $B$-spaces yields
$\tr_B(\ket{\psi}\bra{\psi}) = \sum_{i} \alpha_i^2 \ket{u_i}\bra{u_i}$.
\end{proof}

Let $\ket{0},\ldots,\ket{d-1}$ denote the standard orthonormal basis of $\C^d$.
We consider the discrete Weyl operators $X,Z\in\cL(\C^d)$ defined by
$$
X\ket{i} = \ket{i+1},\quad
Z\ket{i} = \omega^i\,\ket{i},
$$
where $\omega$ denotes a primitive $d$th root of unity and the
addition is modulo~$d$ (see for instance~\citep{ChrWin05}).
We note that $X$ and $Z$ are unitary matrices and $X^{-1}ZX=\omega Z$.

We consider now two copies $\cH_A$ and $\cH_B$ of $\C^d$ and define
the ``maximal entangled state''
$\ket{\psi_{00}}:= \frac1{\sqrt{d}}\sum_\ell \ket{\ell}\ket{\ell}$
of $\cH_A\otimes\cH_B$. By definition, $\ket{\psi_{00}}$ has
the Schmidt coefficients $\frac1{\sqrt{d}}$.
Hence the vectors
$$
\ket{\psi_{ij}} := (\id\otimes X^iZ^j)\ket{\psi_{00}},
$$
obtained from $\ket{\psi_{00}}$ by applying a tensor product of unitary matrices,
have the Schmidt coefficients $\frac1{\sqrt{d}}$ as well.

\begin{lemma}\label{le:due}
The vectors $\ket{\psi_{ij}}$, for $0\le i,j <d$,
form an orthonormal bases of $\cH_A\otimes\cH_B$.
\end{lemma}

\begin{proof}
We have, for some $d$th root of unity $\theta$,
\begin{eqnarray*}
\spr{\psi_{ij}}{\psi_{k\ell}}
 &=& \bra{\psi_{00}} (\id\otimes Z^{-j}X^{-i})(\id\otimes X^kZ^\ell)\ \ket{\psi_{00}} \\
 &=& \theta\, \bra{\psi_{00}} \id\otimes X^{k-i}Z^{\ell-j} \ket{\psi_{00}} \\
 &=& \frac{\theta}{d}\,\sum_{m,m'}\bra{mm} \id\otimes X^{k-i}Z^{\ell-j} \ket{m'm'} \\
 &=& \frac{\theta}{d}\,\sum_{m}\bra{m} X^{k-i}Z^{\ell-j} \ket{m} = \frac{\theta}{d}\,\tr\big( X^{k-i} Z^{\ell-j}\big).
\end{eqnarray*}
It is easy to check that
$\frac{\theta}{d}\,\tr\big( X^{k-i} Z^{\ell-j}\big)=0$
if $\ell\ne j$ or $k\ne i$.
\end{proof}

\begin{proof}[Proof of Proposition~\ref{le:densityoperator}]
Let $r_{ij}$ be the given probability distribution
assuming some bijection $[d^2]\simeq [d]^2$.
According to Lemma~\ref{le:due}, the density operator
$\rho_{AB}:=\sum_{ij} r_{ij} \ket{\psi_{ij}}\bra{\psi_{ij}}$ has
the eigenvalues $r_{ij}$.
Lemma~\ref{le:uno} tells us that
$\tr_B(\ket{\psi_{ij}}\bra{\psi_{ij}})$
has the eigenvalues $1/d$,
hence $\tr_B(\ket{\psi_{ij}}\bra{\psi_{ij}})=\ud_d$.
It follows that
$\tr_B(\rho_{AB})=\ud_d$.
Analogously, we get
$\tr_A(\rho_{AB})=\ud_d$.
\end{proof}

\subsection{Proof of Theorem~\ref{th:main2}}

The proof is essentially the one of Theorem~2 in~\citep{ChrMit06} 
carried out in the special case at hand.
Suppose that $\la\vdash_{d^2}\ell d$.
By Proposition~\ref{le:densityoperator} there is
a density operator~$\rho_{AB}$ having the spectrum $\overline{\la}$
such that
$\tr_A(\rho_{AB}) = \ud_d$,
$\tr_B(\rho_{AB}) = \ud_d$.
Let $P_{X}$ denote the orthogonal projection of
$(\cH_A)^{\otimes k}$ onto the sum of its isotypical components
$\Specht{\mu}\otimes {\Weyl \mu}$ satisfying $\|\overline{\mu} -\ud_d\|_1 \le \e$.
Then
$P_{\overline{X}} := \Id -P_X$ is
the orthogonal projection of
$(\cH_A)^{\otimes k}$ onto the sum of its isotypical components
$\Specht{\mu}\otimes {\Weyl \mu}$ satisfying $\|\overline{\mu} -\ud_d\|_1 > \e$.
The estimation theorem~\eqref{eq:est} implies that
$$
\tr (P_{\overline{X}}\,(\rho_A)^{\otimes k}) \le (k+1)^d (k+1)^{d(d-1)/2}\, e^{-\frac{k}2 \e^2}
 \le (k+1)^{d(d+1)/2}\, e^{-\frac{k}2 \e^2} ,
$$
since there are at most $(k+1)^{d}$ partitions of $k$ of length at most $d$.

Let $P_Y$ denote the orthogonal projection of $(\cH_B)^{\otimes k}$
onto the sum of its isotypical components $\Specht{\nu}\otimes {\Weyl \nu}$ satisfying
$\|\overline{\nu} -u\|_1 \le \e$,
and let $P_{Z}$ denote the orthogonal projection of $(\cH_A\otimes\cH_B)^{\otimes k}$
onto the sum of its isotypical components $\Specht{\Lambda}\otimes {\Weyl \Lambda}$ satisfying
$\|\overline{\Lambda} - \overline{\lambda}\|_1 \le \e$.
We set $P_{\overline{Y}}:=\Id - P_Y$ and $P_{\overline{Z}}:=\Id - P_Z$.
Then we have, similarly as for~$P_X$,
\begin{eqnarray*}
\tr (P_{\overline{Y}}\,(\rho_B)^{\otimes k}) &\le& (k+1)^{d(d+1)/2}\, e^{-\frac{k}2 \e^2}, \\
\tr (P_{\overline{Z}}\,(\rho_{AB})^{\otimes k}) &\le& (k+1)^{d^2(d^2+1)/2}\, e^{-\frac{k}2 \e^2} .
\end{eqnarray*}
By choosing $k=O(\frac{d^4}{\e^2}\log\frac{d}{\e})$ we can achieve that
$$
\tr (P_{\overline{X}}\,(\rho_A)^{\otimes k}) < \frac13,\quad
\tr (P_{\overline{Y}}\,(\rho_{B})^{\otimes k}) < \frac13,\quad
\tr (P_{\overline{Z}}\,(\rho_{AB})^{\otimes k}) < \frac13 .
$$
We put $\sigma:=(\rho_{AB})^{\otimes k}$ in order to simplify notation
and claim that
\begin{equation}\label{claim}
\tr ((P_X\otimes P_Y)\sigma P_Z) > 0 .
\end{equation}
In order to see this, we decompose
$\id=P_X\otimes P_Y + P_{\overline{X}}\otimes \id + P_X\otimes P_{\overline{Y}}$.
From the definition of the partial trace we have
$$
\tr\big((P_{\overline{X}}\otimes \id)\sigma\big)=\tr\big(P_{\overline{X}}(\rho_A)^{\otimes k}\big) < \frac13.
$$
Similarly,
$$
\tr\big((P_X\otimes P_{\overline{Y}})\sigma\big) \le
\tr\big((\id\otimes P_{\overline{Y}})\sigma\big)=\tr\big(P_{\overline{Y}}(\rho_B)^{\otimes k}\big) < \frac13.
$$
Hence
$\tr\big((P_{X}\otimes P_{Y})\sigma\big) > \frac13$.
Using
$\tr\big((P_X\otimes P_{Y})\sigma P_{\overline{Z}}\big) \le \tr(\sigma P_{\overline{Z}}) < \frac13$,
we get
$$
\tr\big((P_X\otimes P_{Y})\sigma P_Z\big)
 = \tr\big((P_X\otimes P_Y)\sigma\big) - \tr\big((P_X\otimes P_Y)\sigma P_{\overline{Z}}\big)
> \frac13 - \frac13 =0 ,
$$
which proves Claim~\eqref{claim}.

Claim~\eqref{claim} implies that
there exist partitions $\mu,\nu,\Lambda$ with
normalizations $\e$-close to $\ud_d,\ud_d,r$, respectively, such that
$(P_\mu\otimes P_\nu)P_{\Lambda}\ne 0$.
Recalling the isotypical decomposition~\eqref{eq:isotyp},
we infer that
$$
(\Specht{\Lambda}\otimes {\Weyl \Lambda}) \cap (\Specht{\mu}\otimes {\Weyl \mu})\otimes (\Specht{\nu}\otimes {\Weyl \nu})
\ne 0 .
$$
Statement~\eqref{eq:split} implies that $g_{\mu,\nu,\Lambda}\ne 0$
and hence the assertion follows for $R_1=\mu,R_2=\nu$. \hfill$\Box$

\subsection{Proof of Lemma~\ref{pro:sKC}}

We assume that $\la,\mu\vdash_d n$. 
The group $\GL_d\times \GL_d \times \GL_d$ operates on 
$\C^d \otimes \C^d \otimes \C^d$ by tensor product, which induces an
action on the polynomial ring~$A$ on $\C^d \otimes \C^d \otimes \C^d$. 
Schur-Weyl duality implies that the submodule $A_n$ of homogeneous polynomials 
of degree~$n$ splits as follows (cf.~\citep{lama:04}):  
$$
 A_n = \bigoplus_{\la, \mu, \nu \vdash_d n} \big(\Specht{\la} \otimes
 \Specht{\mu} \otimes \Specht{\nu} \big)^{S_n} \otimes 
 \mathscr{V}^*_{\la} \otimes  \mathscr{V}^*_{\mu} \otimes  \mathscr{V}^*_{\nu} .
$$ 
We assume now that $g_{\la,\mu,\mu}=\dim \big(\Specht{\la} \otimes
 \Specht{\mu} \otimes \Specht{\mu} \big)^{S_n} \ne 0$ for some $\la,\mu\vdash_d n$. 
Hence there exists a highest weight vector~$F\in A_n$ of weight $(\la,\mu,\mu)$. 
We may assume that the coefficients of $F$ are real (cf.~\citep{BCI:11}). 

Consider the linear automorphism that exchanges the last two factors of 
$\C^d \otimes \C^d \otimes \C^d$. This induces an automorphism~$\sigma$ 
of the algebra~$A$. It is easy to see that 
$F':=\sigma(F)$ is a highest weight vector of weight $(\la,\mu,\mu)$. 
Therefore, both squares $F^2$ and $(F')^2$ are highest weight vectors of weight $(2\la,2\mu,2\mu)$. 
Since $F^2 + (F')^2$ is nonzero and invariant under~$\sigma$, we see that 
$\big(\Specht{2\la} \otimes\Specht{2\mu} \otimes \Specht{2\mu} \big)^{S_n}$ 
has a nonzero invariant with respect to $\sigma$. Hence
$$
 \big(\Specht{2\la} \otimes\Sym^2(\Specht{2\mu})\big)^{S_n} \ne 0 ,
$$ 
which means that 
$sg^{2\la}_{2\mu} \ne 0$. \hfill$\Box$

\end{document}